\documentclass[11pt]{amsart}

\usepackage{amsmath,amsthm,amssymb,mathrsfs}

\usepackage[hidelinks]{hyperref}

\newtheorem{theorem}{Theorem}
\newtheorem{proposition}[theorem]{Proposition}
\newtheorem{corollary}[theorem]{Corollary}
\newtheorem{lemma}[theorem]{Lemma}

\theoremstyle{definition}

\newtheorem{condition}[theorem]{Condition}
\newtheorem{remark}[theorem]{Remark}
\newtheorem{example}[theorem]{Example}

\begin{document}

\title[Self-similar solutions to negative power curvature flow]{Closed self-similar solutions to flows by negative powers of curvature}
\author{Shanze Gao}
\address{School of Mathematics and Statistics \\  Shaanxi Normal University \\ Xi'an 710119, P. R. China}
\email{gaoshanze@snnu.edu.cn}


\begin{abstract}
In some warped product manifolds including space forms, we consider closed self-similar solutions to curvature flows whose speeds are negative powers of mean curvature, Gauss curvature and other curvature functions with suitable properties. We prove such self-similar solutions, not necessarily strictly convex for some cases, must be slices of warped product manifolds. A new auxiliary function is the key of the proofs.
\end{abstract}

\maketitle

\section{Introduction}\label{sec1}

In this paper, a closed, immersed hypersurface $X:M^{n}\rightarrow \mathbb{R}^{n+1}$ ($n\geq 2$) satisfying equation
\begin{equation*}
F^{\beta}(\mathcal{W}(x))=\langle X(x),\nu(x) \rangle 
\end{equation*}
is called a \emph{closed self-similar solution}, where $F$ is a suitable function of the shape operator $\mathcal{W}$ of $M^{n}$, $\beta\neq 0$ is a constant and $\nu$ denotes the unit outward normal vector. In fact, this hypersurface is corresponding to the self-similar solution to curvature flow which satisfies \[\partial_{t}X=-\mathrm{sign}(\beta)F^{\beta}\nu.\] 

Flow by negative powers of curvature (also called inverse curvature flow) refers to the flow with $\beta<0$ and $F(\mathcal{W})=f(\kappa(\mathcal{W}))$ where $f$ is a $1$-homogeneous (homogeneous of degree $1$) function of principal curvatures $\kappa=(\kappa_{1},\dots,\kappa_{n})$. This type of flow has been studied by many researchers (see \cite{Gerhardt90,Urbas90,Urbas91,Gerhardt14,Li-Wang-Wei19,Wei,Jin-Wang-Wei} etc.). Urbas \cite{Urbas99} also considered noncompact self-similar solutions to flow by negative powers of Gauss curvature.

Self-similar solutions to mean curvature flow and Gauss curvature flow have been widely studied. Although closed self-similar solution is not unique in general (see example of Angenent \cite{Angenent}), it must be a sphere provided suitable conditions. Huisken \cite{Huisken90} showed that any closed self-similar solution to mean curvature flow (also known as self-shrinker) must be a sphere if mean curvature is nonnegative. Colding-Minicozzi \cite{Colding-Minicozzi12} introduced a variational characterization of self-shrinker and proved spheres are the only closed $F$-stable self-shrinkers. For flow by powers of Gauss curvature, Brendle-Choi-Daskalopoulos \cite{BCD17} proved closed and strictly convex self-similar solutions must be spheres when the power $\alpha>1/(n+2)$. For curvature flow with general $F$ and positive power $\beta$, uniqueness results of closed self-similar solution were obtained by McCoy \cite{McCoy11,McCoy21}, Gao-Li-Ma \cite{GLM18} and Gao-Li-Wang \cite{GLW22} etc.

The concept of self-similar solution can be extended to hypersurfaces in warped product manifolds. Let $\overline{M}^{n+1}=[0,\bar{r})\times_{\lambda} N^{n}$ be a warped product manifold with metric \[\bar{g}=dr\otimes dr+\lambda^{2}(r)g_{\scriptscriptstyle N},\]
where $(N^{n},g_{\scriptscriptstyle N})$ is a closed Riemannian manifold and $\lambda(r)$ is a smooth, positive function of $r\in (0,\bar{r})$. It is known that $\lambda(r)\partial_{r}$ is a conformal vector field on $\overline{M}^{n+1}$. A hypersurface $M^{n}$ in $\overline{M}^{n+1}$ is also called a self-similar solution if it satisfies
\begin{equation*}
F^{\beta}(\mathcal{W}(x))=\bar{g}\big(\lambda(r(x))\partial_{r}(x),\nu(x)\big).
\end{equation*}
In fact, we can generate a family of hypersurfaces by $M^{n}$ and $\lambda\partial_{r}$ which satisfies the equation of corresponding curvature flow (see \cite{Alias-deLira-Rigoli20,GM19} for details). The ambient space $\overline{M}$ is actually Euclidean space $\mathbb{R}^{n+1}$, sphere $\mathbb{S}^{n+1}$ or hyperbolic space $\mathbb{H}^{n+1}$ if $N^{n}=\mathbb{S}^{n}$ and $\lambda(r)=r,\sin r\text{ or }\sinh r$ correspondingly.

Self-shrinkers in warped product manifolds were studied by Wu \cite{WuG17}, Alias-de Lira-Rigoli \cite{Alias-deLira-Rigoli20} etc. Ma and the author \cite{GM19} considered closed self-similar solutions to flow with some general $F$ and positive power $\beta$ in warped product manifolds. And they proved uniqueness of solutions if the ambient space is a hemisphere. Later, Gao-Li-Wang \cite{GLW22} extended the uniqueness result to more general $F$ which is an inverse concave function of principal curvatures.

For flow by negative powers of the $k$-th mean curvature ($k<n$), Ma and the author \cite{GM21} proved uniqueness of closed, strictly convex self-similar solutions in a class of warped product manifolds. 

In this paper, we consider closed self-similar solutions which are not necessarily strictly convex. A hypersurface is called \emph{mean convex} if its mean curvature $H=\frac{1}{n}(\kappa_{1}+\cdots+\kappa_{n})$ is positive everywhere. 

\begin{theorem}\label{Th:H}
Suppose that $\overline{M}^{n+1}=[0,\bar{r})\times_{\lambda} N^{n}$ is a warped product manifold satisfying
\begin{equation*}
\mathrm{Ric}_{\scriptscriptstyle N}\geq (n-1)(\lambda'^{2}-\lambda\lambda'')g_{\scriptscriptstyle N},
\end{equation*}
and $M^{n}$ be a closed, immersed hypersurface in $\overline{M}^{n+1}$. If $M^{n}$ is mean convex and satisfies
\begin{equation*}
H^{-\alpha}=\bar{g}(\lambda\partial_{r},\nu),
\end{equation*}
where $\alpha>0$ is a constant, then $M^{n}$ is a slice $\{ r_{0}\}\times N^{n} $ for some $r_{0}\in (0,\bar{r})$. 
\end{theorem}

\begin{remark}
Compared with Corollary 2 in \cite{GM21}, instead of strictly convex, $M^{n}$ is mean convex in the above theorem. It is easy to check that $\mathbb{R}^{n+1}$, $\mathbb{S}^{n+1}$ and $\mathbb{H}^{n+1}$ satisfy the assumption of $\overline{M}$.
\end{remark}

Let $\Gamma\subset\mathbb{R}^{n}$ be an open, convex, symmetric cone with vertex at the origin, which contains the positive cone $\Gamma_{+}=\{(\kappa_1,\ldots,\kappa_n)\in \mathbb{R}^n:~\kappa_i>0~\text{ for any }~i=1,\dots,n\}$.

\begin{condition}\label{condtn}
We assume that $F(\mathcal{W})=f(\kappa(\mathcal{W}))$ satisfies the following properties in $\Gamma$.
\begin{itemize}
\item[(i)] $f$ is a smooth, symmetric function of the eigenvalues $\kappa$ of $\mathcal{W}$.
\item[(ii)] $f$ is positive in $\Gamma$ and normalized such that $f(1,\dots,1)=1$.
\item[(iii)] $f$ is strictly increasing in each argument, i.e., $\frac{\partial f}{\partial\kappa_{i}}>0$ in $\Gamma$ for any $i=1,\dots,n$.
\item[(iv)] $f$ is $1$-homogeneous (homogeneous of degree $1$), i.e., $f(s\kappa)=sf(\kappa)$ for any $s>0$ and $\kappa\in\Gamma$.
\item[(v)] The following inequalities hold in $\Gamma$:
\begin{equation*}
\sum_{i=1}^{n}\frac{\partial f}{\partial\kappa_{i}}\geq 1,\qquad \sum_{i=1}^{n}\frac{\partial f}{\partial\kappa_{i}}\kappa_{i}^{2}\geq f^{2},
\end{equation*}
and $f-\frac{\partial f}{\partial\kappa_{i}}\kappa_{i}\geq 0$ for any $i=1,\dots,n$.
\end{itemize}
\end{condition}

\begin{theorem}\label{Th:mainthm}
Suppose that $\overline{M}^{n+1}=[0,\bar{r})\times_{\lambda} N^{n}$ is a warped product manifold,  
where $(N,g_{\scriptscriptstyle N})$ is a closed Riemannian manifold with constant sectional curvature $c$ and $\lambda(r)$ satisfies $\lambda'(r)>0$ and
\begin{equation}\label{Eq:C4>=}
\frac{\lambda(r)''}{\lambda(r)}+\frac{c-\lambda(r)'^{2}}{\lambda(r)^{2}}\geq 0.
\end{equation}
Let $M^{n}$ be a closed, immersed hypersurface in $\overline{M}^{n+1}$
satisfying
\begin{equation*}
F^{-\alpha}=\bar{g}(\lambda\partial_{r},\nu),
\end{equation*}
where $\alpha>0$ is a constant. If principal curvatures $\kappa$ of $M^{n}$ are in a cone $\Gamma$ such that Condition \ref{condtn} holds, then $M^{n}$ is a slice $\{r_{0}\}\times N^{n}$ for some $r_{0}\in (0,\bar{r})$.
\end{theorem}

\begin{remark}
The assumption of the ambient space $\overline{M}$ in the above theorem is stronger than it in Theorem \ref{Th:H}. In fact, inequality \eqref{Eq:C4>=} implies \begin{align*}
(\lambda'^{2}-\lambda\lambda'')g_{\scriptscriptstyle N}\leq cg_{\scriptscriptstyle N}=\frac{1}{n-1}\mathrm{Ric}_{\scriptscriptstyle N}.
\end{align*}
It can be checked that space forms $\mathbb{R}^{n+1}$, $\mathbb{S}^{n+1}$ and $\mathbb{H}^{n+1}$ still satisfy the assumption and readers may refer to \cite{Brendle13} for more spaces satisfying \eqref{Eq:C4>=}.
\end{remark}

Let $\sigma_{k}(\kappa)$ denote the $k$-th elementary symmetric polynomial of principal curvatures $\kappa=(\kappa_{1},...,\kappa_{n})$ of hypersurface $M^{n}$, i.e.,
\begin{equation*}
\sigma_{k}(\kappa)=\sum_{1\leq i_{1}<...<i_{k}\leq n}\kappa_{i_{1}}\cdots\kappa_{i_{k}}.
\end{equation*}
The $k$-th mean curvature is defined by $H_{k}=\sigma_{k}(\kappa)/\binom{n}{k}$.

The following corollary of Theorem \ref{Th:mainthm} holds under weaker convexity assumption of $M^{n}$ (compared with Corollary 5 in \cite{GM21}).
\begin{corollary}\label{Th:Hk}
Suppose $\overline{M}^{n+1}$ is the same in Theorem \ref{Th:mainthm} and $M^{n}$ is a closed, immersed hypersurface in $\overline{M}^{n+1}$. If $M^{n}$ satisfies  $H_{k+1}>0$ and
\begin{equation*}
H_{k}^{-\alpha}=\bar{g}(\lambda\partial_{r},\nu),
\end{equation*}
where $\alpha>0$ is a constant and $2\leq k\leq n-1$, then $M^{n}$ is a slice $\{r_{0}\}\times N^{n}$ for some $r_{0}\in (0,\bar{r})$.
\end{corollary}

If $M^{n}$ is strictly convex, we have the following corollaries from Theorem \ref{Th:mainthm}.

\begin{corollary}\label{Th:K}
Suppose $\overline{M}^{n+1}$ is under the same assumption of Theorem \ref{Th:mainthm} and $M^{n}$ is a closed, immersed hypersurface in $\overline{M}^{n+1}$. If $M^{n}$ is strictly convex and satisfies
\begin{equation*}
K^{-\alpha}=\bar{g}(\lambda\partial_{r},\nu),
\end{equation*}
where $K$ is Gauss curvature and $\alpha>0$ is a constant, then $M^{n}$ is a slice $\{r_{0}\}\times N^{n}$ for some $r_{0}\in (0,\bar{r})$.
\end{corollary}

We say that $f$ is \emph{inverse concave} if the function
\begin{equation*}
f^{*}(\kappa_{1},\dots,\kappa_{n}):=\frac 1{f(\frac{1}{\kappa_{1}},\dots,\frac{1}{\kappa_{n}})}
\end{equation*}
is concave. It is known that a class of functions are concave and inverse concave in \cite{Andrews07}, for example $(\frac{\sigma_{k}}{\sigma_{l}})^{\frac{1}{k-l}}$ where $\sigma_{k}$ is the $k$-th elementary polynomials and $0\leq l<k\leq n$. These convexity conditions of functions appear naturally in the study of curvature flows and other fully nonlinear PDEs (see, for example, \cite{Andrews07,Bian-Guan09,McCoy11}).

\begin{corollary}\label{Th:F}
Suppose $\overline{M}^{n+1}$ is under the same assumption of Theorem \ref{Th:mainthm} and $M^{n}$ is a closed, immersed hypersurface in $\overline{M}^{n+1}$. If $M^{n}$ is strictly convex and satisfies
\begin{equation*}
F^{-\alpha}=\bar{g}(\lambda\partial_{r},\nu),
\end{equation*}
where $F$ is concave, inverse concave and satisfies Condition \ref{condtn} (i)-(iv), $\alpha>0$ is a constant, then $M^{n}$ is a slice $\{r_{0}\}\times N^{n}$ for some $r_{0}\in (0,\bar{r})$.
\end{corollary}

There are some connections between self-similar solutions and hypersurfaces of constant curvatures. Brendle \cite{Brendle13} proves closed and embedded hypersurfaces with constant mean curvature in a class of warped product manifolds must be umbilic. The case of hypersurfaces with constant $H_{k}$ is showed by Brendle and Eichmair \cite{Brendle-Eichmair13}. In these papers, Heintze-Kacher type inequality and Minkowski type formula in warped product manifolds are established, which can also be used to obtain uniqueness of self-similar solutions (see \cite{GM21}). Rigidity problems of hypersurfaces in warped product manifolds are also considered in \cite{Wu-Xia,Kwong-Lee-Pyo} etc.

Now, we briefly recall the methods in \cite{GM21} and compare them with the proofs in this paper. For example, let us consider the case $\overline{M}^{n+1}=\mathbb{R}^{n+1}$ and $M^{n}$ satisfies
\begin{equation}\label{Eq:H}
H^{-\alpha}=\langle X,\nu \rangle, 
\end{equation}
where $H=\frac{1}{n}(\kappa_{1}+...+\kappa_{n})$ is normalized mean curvature.

One method in \cite{GM21} is based on an integral inequality
\begin{equation}\label{Eq:0}
0=\int_{M^{n}}\mathrm{div}(HX^{T}-\nabla \langle X,\nu \rangle)\geq-(n-1)\int_{M^{n}}\langle X,\nabla H \rangle.
\end{equation}
If $M^{n}$ is strictly convex, using \eqref{Eq:H}, we know
\begin{equation}\label{Eq:X}
\langle X,\nabla H \rangle=-\frac{1}{\alpha}\langle X,\nu \rangle^{-\frac{\alpha+1}{\alpha}}\sum_{i}\kappa_{i}\langle X,e_{i} \rangle^{2}\leq 0.
\end{equation}
Combing \eqref{Eq:0} and \eqref{Eq:X}, we know $M^{n}$ is a sphere.

The other method uses the Heintze-Karcher inequality for embedded and mean convex $M^{n}$
\begin{equation*}
\int_{M^{n}}\langle X,\nu \rangle \leq \int_{M^{n}}\frac{1}{H}
\end{equation*}
and the Minkowski formula
\begin{equation*}
|M^{n}|=\int_{M^{n}}H\langle X,\nu \rangle.
\end{equation*}

Combining with \eqref{Eq:H}, we have
\begin{equation}\label{Eq:HK}
\int_{M^{n}}H^{-\alpha} \leq \int_{M^{n}}H^{-1}
\end{equation}
and
\begin{equation}\label{Eq:M}
|M^{n}|=\int_{M^{n}}H^{1-\alpha}.
\end{equation}
If $\alpha>1$, using H\"{o}lder inequality, we obtain
\begin{align*}
\int_{M^{n}}H^{1-\alpha}\int_{M^{n}}H^{-1}\leq |M_{n}|\int_{M^{n}}H^{-\alpha}.
\end{align*}
From \eqref{Eq:HK} and \eqref{Eq:M} we know that equality occurs in the above inequality. This implies $M^{n}$ is a sphere.

However, it seems that these integral methods can not be generalized for general curvature function $F$, except for the case $F=H_{k}$. This motivates us to seek a proof via the maximum principle.

In fact, for \eqref{Eq:H}, we introduce an auxiliary quantity
\begin{equation*}
P=\frac{|X|^{2}}{2}-\frac{\alpha}{\alpha+1}\langle X,\nu \rangle^{\frac{\alpha+1}{\alpha}},
\end{equation*}
which is similar to Weinberger's $P$-function \cite{Weinberger} in spirit. We notice
\begin{align*}
\Delta P&=\langle X,\nu \rangle^{\frac{\alpha+1}{\alpha}}(|h|^{2}-nH^{2})+\frac{1}{\alpha}\langle X,\nu \rangle^{\frac{1-\alpha}{\alpha}}\sum_{i}(nH-\kappa_{i})\kappa_{i}\langle X,e_{i} \rangle^{2} \\
&\geq \frac{1}{\alpha}\langle X,\nu \rangle^{\frac{1-\alpha}{\alpha}}\sum_{i}(nH-\kappa_{i})\kappa_{i}\langle X,e_{i} \rangle^{2},
\end{align*}
where $|h|^{2}=\sum_{i}\kappa_{i}^{2}$ and the last step is from the Cauchy-Schwarz inequality. Thus, function $P$ is subharmonic if $M^{n}$ is strictly convex, which implies $M^{n}$ is a sphere. Proofs of Theorem \ref{Th:H} and \ref{Th:mainthm} are based on this observation.

The paper is organized as follows. In Section \ref{Sec:prelim}, we show some examples of functions satisfying Condition \ref{condtn} and recall some facts of hypersurfaces in warped product manifolds. In Section \ref{Sec:P}, we derive a basic formula of auxiliary function $P$. In Section \ref{Sec:pfH}, we present the proof of Theorem \ref{Th:H}. In Section \ref{Sec:MainTHM}, we prove Theorem \ref{Th:mainthm} and its corollaries.

\section{Preliminaries}\label{Sec:prelim}

Throughout this paper, repeated indexes will be added up from $1$ to $n$ unless otherwise stated.

\subsection{Function $F$ and its defining cone $\Gamma$}

We recall some facts of symmetric functions for later calculations (see \cite{Gerhardt96,Spruck05} for example).

If matrix $\mathcal{W}=(h_{ij})$ is diagonal, i.e., $h_{ij}=\kappa_{i}\delta_{ij}$ for any $1\leq i,j\leq n$, then the following formula holds for function $F(\mathcal{W})=f(\kappa(\mathcal{W}))$:
\begin{equation*}
\frac{\partial F}{\partial h_{ij}}b_{ij}=\frac{\partial f}{\partial\kappa_{p}}b_{pp},\quad \text{for any symmetric matrix } B=(b_{ij}).
\end{equation*}

If function $f=f(\kappa)$ is $1$-homogeneous, differentiating $f(s\kappa)=sf(\kappa)$ with respect to $s$ gives $\frac{\partial f}{\partial\kappa_{i}}\kappa_{i}=f$.

Let $\sigma_{k}(\kappa)$ denote the $k$-th elementary symmetric polynomial of $\kappa=(\kappa_{1},...,\kappa_{n})$, i.e.,
\begin{equation*}
\sigma_{k}(\kappa)=\sum_{1\leq i_{1}<...<i_{k}\leq n}\kappa_{i_{1}}\cdots\kappa_{i_{k}}.
\end{equation*}
And let $\sigma_{k}(\kappa|i)$ denote $\sigma_{k}(\kappa)$ with $\kappa_{i}=0$ for a fixed $i$. We also set $\sigma_{0}(\kappa)=1$ and $\sigma_{k}(\kappa)=0$ if $k>n$ or $k<0$.

The G{\aa}rding's cone is defined by
\begin{equation*}
\Gamma_{k}:=\{\kappa\in\mathbb{R}^{n}|\sigma_{m}(\kappa)>0 \text{ for any }1\leq m\leq k\}.
\end{equation*}
We also consider the following cone
\begin{equation*}
\widetilde{\Gamma}_{k}:=\left\{\kappa\in\mathbb{R}^{n}\left|\begin{aligned}
&\sigma_{m}(\kappa)>0 \text{ for any }1\leq m\leq k-1,\\
&\sigma_{k}(\kappa|i)>-(k-1)\sigma_{k}(\kappa) \text{ for any } 1\leq i\leq n
\end{aligned}\right.\right\},
\end{equation*}
where $1\leq k\leq n-1$.

Combining $\sum_{i=1}^{n}\sigma_{k}(\kappa|i)=(n-k)\sigma_{k}(\kappa)$, inequalities $\sigma_{k}(\kappa|i)>-(k-1)\sigma_{k}(\kappa)$ implies $\sigma_{k}(\kappa)>0$ which means $\widetilde{\Gamma}_{k} \subset \Gamma_{k}$. Furthermore, noticing $\sigma_{k}(\kappa|i)>0$ in $\Gamma_{k+1}$, we know $\widetilde{\Gamma}_{k}$ has the following relation with G{\aa}rding's cones: $\Gamma_{k+1}\subset \widetilde{\Gamma}_{k} \subset \Gamma_{k}\subset \widetilde{\Gamma}_{k-1}$.

Now we show some examples of $F$ satisfying Condition \ref{condtn}.
\begin{example}\label{Ex:Hk}
Function $H_{k}^{\frac{1}{k}}(\kappa)$ in $\widetilde{\Gamma}_{k}$, where $H_{k}(\kappa)=\frac{1}{\binom{n}{k}}\sigma_{k}(\kappa)$ and $1\leq k\leq n-1$.
\end{example}

We only show Condition \ref{condtn} (v) since (i)-(iv) are easy to be checked. If $f=H_{k}^{\frac{1}{k}}$, then
\begin{align*}
\sum_{i=1}^{n}\frac{\partial f}{\partial \kappa_{i}}=\frac{1}{k}H_{k}^{\frac{1-k}{k}}\sum_{i=1}^{n}\frac{\partial H_{k}}{\partial\kappa_{i}}=H_{k}^{\frac{1-k}{k}}H_{k-1}\geq 1,
\end{align*}
where the last inequality is from MacLaurin inequality. And
\begin{align*}
\sum_{i=1}^{n}\frac{\partial f}{\partial\kappa_{i}}\kappa_{i}^{2}&=\frac{1}{k}H_{k}^{\frac{1-k}{k}}\sum_{i=1}^{n}\frac{\partial H_{k}}{\partial\kappa_{i}}\kappa_{i}^{2} \\
&=\frac{n}{k}H_{1}H_{k}^{\frac{1}{k}}-\frac{n-k}{k}H_{k}^{\frac{1-k}{k}}H_{k+1}. 
\end{align*}
From inequality $H_{1}H_{k}\geq H_{k+1}$ and MacLaurin inequality, we know
\begin{align*}
\sum_{i=1}^{n}\frac{\partial f}{\partial\kappa_{i}}\kappa_{i}^{2}\geq H_{1}H_{k}^{\frac{1}{k}} \geq H_{k}^{\frac{2}{k}}=f^{2}.
\end{align*}
We notice, for any index $1\leq i\leq n$ fixed,
\begin{align*}
f-\frac{\partial f}{\partial\kappa_{i}}\kappa_{i}=H_{k}^{\frac{1}{k}}\left(1-\frac{\sigma_{k-1}(\kappa|i)\kappa_{i}}{k\sigma_{k}(\kappa)}\right).
\end{align*}
From $\sigma_{k}(\kappa)=\sigma_{k}(\kappa|i)+\sigma_{k-1}(\kappa|i)\kappa_{i}$ and the definition of $\widetilde{\Gamma}_{k}$, we know
\begin{align*}
1-\frac{\sigma_{k-1}(\kappa|i)\kappa_{i}}{k\sigma_{k}(\kappa)}=\frac{(k-1)\sigma_{k}(\kappa)+\sigma_{k}(\kappa|i)}{\sigma_{k}(\kappa)}>0.
\end{align*}
Thus, we know $H_{k}^{\frac{1}{k}}(\kappa)$ in $\widetilde{\Gamma}_{k}$ satisfies Condition \ref{condtn}.

\begin{remark}
For function $H_{k}^{\frac{1}{k}}$, almost all requirements in Condition \ref{condtn} hold in $\Gamma_{k}$ except the inequality $f-\frac{\partial f}{\partial\kappa_{i}}\kappa_{i}\geq 0$ for any $1\leq i\leq n$. In fact, there are examples showing that $\Gamma_{k}$ is not sufficient. For the case $n=3$ and $k=2$, if $\kappa_{1}=-\frac{1}{2}$, $\kappa_{2}=1$ and $\kappa_{3}=\frac{3}{2}$, we can check such $\kappa\in \Gamma_{2}$ but $f-\frac{\partial f}{\partial\kappa_{3}}\kappa_{3}<0$ for $f=H_{2}^{\frac{1}{2}}$.
\end{remark}

\begin{example}\label{Ex:sigman}
Function $\sigma_{n}^{\frac{1}{n}}(\kappa)$ in $\Gamma_{n}=\Gamma_{+}$.
\end{example}

Inequalities $\sum_{i=1}^{n}\frac{\partial f}{\partial \kappa_{i}}\geq 1$ and $\sum_{i=1}^{n}\frac{\partial f}{\partial\kappa_{i}}\kappa_{i}^{2}\geq f^{2}$ can be checked similar to Example \ref{Ex:Hk}. If the defining cone $\Gamma=\Gamma_{+}$, $f-\frac{\partial f}{\partial\kappa_{i}}\kappa_{i}\geq 0$ follows from \[ f-\frac{\partial f}{\partial\kappa_{i}}\kappa_{i}=\sum_{j=1}^{n}\frac{\partial f}{\partial\kappa_{j}}\kappa_{j}-\frac{\partial f}{\partial\kappa_{i}}\kappa_{i}=\sum_{j\neq i}\frac{\partial f}{\partial\kappa_{j}}\kappa_{j}>0. \]

\begin{example}\label{Ex:candinc}
Concave and inverse concave function $F$ in $\Gamma_{+}$ which satisfies Condition \ref{condtn} (i)-(iv).
\end{example}

We only need to ensure inequalities $\sum_{i=1}^{n}\frac{\partial f}{\partial \kappa_{i}}\geq 1$ and $\sum_{i=1}^{n}\frac{\partial f}{\partial\kappa_{i}}\kappa_{i}^{2}\geq f^{2}$. The proofs can be found in \cite[Lemma 4 and 5]{Andrews-McCoy-Zheng} and we show them for the readers' convenience. Since $f$ is concave and $1$-homogeneous, we have
\begin{align*}
1=f(1,...,1)\leq f(\kappa)+\frac{\partial f}{\partial \kappa_{i}}(\kappa)(1-\kappa_{i})=\sum_{i}\frac{\partial f}{\partial \kappa_{i}}(\kappa).
\end{align*}
Let $\kappa^{*}=(\kappa_{1}^{*},\dots,\kappa_{n}^{*})$ where $\kappa_{i}^{*}=\frac{1}{\kappa_{i}}$ for $1\leq i\leq n$. Since $f$ is also inverse concave, i.e., $f^{*}(\kappa_{1},\dots,\kappa_{n})=\frac{1}{f(\kappa_{1}^{*},\dots,\kappa_{n}^{*})}$ is concave, we have
\begin{align*}
1\leq \sum_{i}\frac{\partial f^{*}}{\partial\kappa_{i}}(\kappa^{*})=\frac{1}{(f(\kappa))^{2}}\frac{\partial f}{\partial\kappa_{i}}(\kappa)\kappa_{i}^{2}.
\end{align*}

The following proposition shows that the class of functions satisfying Condition \ref{condtn} has some convexity properties.
\begin{proposition}\label{Th:ConvexClass}
If $f_{1}$ and $f_{2}$ satisfy Condition \ref{condtn} in the same defining cone $\Gamma\supset \Gamma_{+}$, then $\lambda f_{1}+(1-\lambda)f_{2}$ and $f_{1}^{\lambda}f_{2}^{1-\lambda}$ also satisfy Condition \ref{condtn} in $\Gamma$ for any $\lambda\in [0,1]$.
\end{proposition}

\begin{proof}
Denote $f=\lambda f_{1}+(1-\lambda)f_{2}$ and $\tilde{f}=f_{1}^{\lambda}f_{2}^{1-\lambda}$.

First, we show $f$ satisfies Condition \ref{condtn}. It is clear for (i)-(iv), so we only check (v). By direct calculations, we know
\begin{align*}
\sum_{i=1}^{n}\frac{\partial f}{\partial\kappa_{i}}=\lambda\sum_{i=1}^{n}\frac{\partial f_{1}}{\partial\kappa_{i}}+(1-\lambda)\sum_{i=1}^{n}\frac{\partial f_{2}}{\partial\kappa_{i}}\geq \lambda+1-\lambda=1
\end{align*}
and
\begin{align*}
f-\frac{\partial f}{\partial\kappa_{i}}\kappa_{i}=\lambda\left(f_{1}-\frac{\partial f_{1}}{\partial\kappa_{i}}\kappa_{i}\right)+(1-\lambda)\left(f_{2}-\frac{\partial f_{2}}{\partial\kappa_{i}}\kappa_{i}\right)\geq 0.
\end{align*}
And
\begin{align*}
\sum_{i=1}^{n}\frac{\partial f}{\partial\kappa_{i}}\kappa_{i}^{2}&=\lambda\sum_{i=1}^{n}\frac{\partial f_{1}}{\partial\kappa_{i}}\kappa_{i}^{2}+(1-\lambda)\sum_{i=1}^{n}\frac{\partial f_{2}}{\partial\kappa_{i}}\kappa_{i}^{2} \\
&\geq \lambda f_{1}^{2}+(1-\lambda)f_{2}^{2} \geq f^{2},
\end{align*}
where we use Jensen's inequality in the last step.

Next, we consider $\tilde{f}$. It is obvious for (i), (ii) and (iv). From \[ \frac{\partial\tilde{f}}{\partial\kappa_{i}}=\tilde{f}\left( \frac{\lambda}{f_{1}}\frac{\partial f_{1}}{\partial\kappa_{i}}+\frac{1-\lambda}{f_{2}}\frac{\partial f_{2}}{\partial\kappa_{i}} \right), \]
we see \[ \frac{\partial\tilde{f}}{\partial\kappa_{i}}>0 \quad\text{ and }\quad \tilde{f}-\frac{\partial\tilde{f}}{\partial\kappa_{i}}=\tilde{f}\left( \frac{1-\lambda}{f_{1}}\frac{\partial f_{1}}{\partial\kappa_{i}}+\frac{\lambda}{f_{2}}\frac{\partial f_{2}}{\partial\kappa_{i}} \right)>0 \] for any $i=1,\dots,n$. Furthermore,
\begin{align*}
\sum_{i=1}^{n}\frac{\partial\tilde{f}}{\partial\kappa_{i}}\geq \lambda \left( \frac{f_{2}}{f_{1}} \right)^{1-\lambda}+(1-\lambda)\left( \frac{f_{1}}{f_{2}} \right)^{\lambda}\geq 1
\end{align*}
and
\begin{align*}
\sum_{i=1}^{n}\frac{\partial\tilde{f}}{\partial\kappa_{i}}\kappa_{i}^{2} \geq \tilde{f}\big(\lambda f_{1}+(1-\lambda)f_{2}\big) \geq \tilde{f}^{2},
\end{align*}
where Young's inequality is used in the last steps of above inequalities.
\end{proof}

Proposition \ref{Th:ConvexClass} implies more examples satisfying Condition \ref{condtn} on suitable cones larger than the positive cone. For instance, $\lambda H_{k-1}+(1-\lambda)H_{k}$ and $H_{k-1}^{\lambda}H_{k}^{1-\lambda}$ both satisfy Condition \ref{condtn} in $\widetilde{\Gamma}_{k}$ for any $\lambda\in (0,1)$.

\subsection{Hypersurface in warped product manifold}
Let $M^{n}$ be a hypersurface in $\overline{M}^{n+1}$. We will calculate under an orthonormal frame $\{ e_{1},\dots,e_{n} \}$ on $M^{n}$ in this paper. Let $(h_{ij})$ denotes the second fundamental form. Under the orthonormal frame, $(h_{ij})$ equals to the matrix of shape operator $\mathcal{W}$. Then
\begin{equation*}
\nabla_{i}\nu=h_{ij}e_{j}.
\end{equation*}
By Codazzi equation,
\begin{equation}
\nabla_{i}h_{jl}=\nabla_{l}h_{ij}+\bar{R}_{\nu jli},
\end{equation}
where $\bar{R}_{\nu jli}=\bar{R}(\nu,e_{i},e_{l},e_{j})$ and $\bar{R}(\cdot,\cdot,\cdot,\cdot)$ is the $(0,4)$-Riemannian curvature tensor of $\overline{M}^{n+1}$.

Now we assume $\overline{M}^{n+1}=[0,\bar{r})\times_{\lambda} N^{n}$ is a warped product manifold. Let $\bar{\nabla}$ denote the Levi-Civita connection of $\overline{M}^{n+1}$. The vector field $\lambda(r)\partial_{r}$ satisfies
\begin{equation*}
\bar{\nabla}_{\xi}\lambda(r)\partial_{r}=\lambda'(r)\xi
\end{equation*}
for any vector field $\xi$ on $\overline{M}^{n+1}$. 

Define $\Phi(r):=\int_{0}^{r}\lambda(s)ds$. Then $\bar{\nabla}\Phi=\lambda\partial_{r}$.

\begin{lemma}\label{Th:Rivnu}
If $\overline{M}^{n+1}=[0,\bar{r})\times_{\lambda} N^{n}$ satisfies
\begin{equation*}
\mathrm{Ric}_{\scriptscriptstyle N}\geq (n-1)(\lambda'^{2}-\lambda\lambda'')g_{\scriptscriptstyle N},
\end{equation*}
then $\overline{\mathrm{Ric}}(\nu,\lambda\partial_{r}^{T})\leq 0$, where $\partial_{r}^{T}=\partial_{r}-\bar{g}(\partial_{r},\nu)\nu$.
\end{lemma}

\begin{proof}
A proof can be found in \cite[ Page 699]{GM21}.
\end{proof}

If we further assume $(N,g_{\scriptscriptstyle N})$ has constant sectional curvature $c$. Then
\begin{equation}\label{Eq:Rnu}
\bar{R}_{\nu jli}=-\left(\frac{\lambda''}{\lambda}+\frac{c-\lambda'^{2}}{\lambda^{2}}\right)(\delta_{ij}r_{l}-\delta_{jl}r_{i})r_{\nu},
\end{equation}
where $r_{i}=\bar{g}(\partial_{r},e_{i})$ and $r_{\nu}=\bar{g}(\partial_{r},\nu)$ (see \cite{ONeill}, or \cite{GM21} for convenience).

\section{Auxiliary function $P$}\label{Sec:P}

In this section, we assume that hypersurface $M^{n}$ in $\overline{M}^{n+1}=[0,\bar{r})\times_{\lambda} N^{n}$ satisfies
\begin{equation*}
F^{-\alpha}=\bar{g}(\lambda\partial_{r},\nu).
\end{equation*}
Denote $u:=\bar{g}(\lambda\partial_{r},\nu)$. And we still use $\Phi$ to denote its pull-back on $M^{n}$ by the immersion $M^{n}\rightarrow \overline{M}^{n+1}$. 

By direct calculation,
\begin{align}
\nabla_{j}u&=h_{jl}\bar{g}(\lambda\partial_{r},e_{l}), \label{Eq:Du} \\
\nabla_{i}\nabla_{j}u&=\nabla_{i}h_{jl}\bar{g}(\lambda\partial_{r},e_{l})+\lambda'h_{ij}-h_{il}h_{jl}\bar{g}(\lambda\partial_{r},\nu) \label{Eq:DDu} \\
&=\nabla_{l}h_{ij}\bar{g}(\lambda\partial_{r},e_{l})+\bar{R}_{\nu jli}\bar{g}(\lambda\partial_{r},e_{l})+\lambda'h_{ij}-uh_{il}h_{jl}.\nonumber
\end{align}

It is also easy to check that
\begin{align}
\nabla_{j}\Phi&=\bar{g}(\lambda\partial_{r},e_{j}),\\
\nabla_{i}\nabla_{j}\Phi&=\lambda'\delta_{ij}-h_{ij}\bar{g}(\lambda\partial_{r},\nu)=\lambda'\delta_{ij}-uh_{ij}, \label{Eq:DDPhi}
\end{align}
where $\delta_{ij}$ is the Kronecker symbol.

Define operator $\mathcal{L}:=F^{ij}\nabla_{i}\nabla_{j}$, where $F^{ij}=\frac{\partial F}{\partial h_{ij}}$. We consider the following auxiliary function \[ P:=\Phi-\frac{\alpha}{\alpha+1}u^{\frac{\alpha+1}{\alpha}}. \] Here $u>0$ is confirmed by $F^{-\alpha}=u$ and assumption of $F>0$ in Theorem \ref{Th:H} and \ref{Th:mainthm}.

\begin{lemma}\label{Th:LP}
Function $P$ satisfies the following equality:
\begin{align*}
\mathcal{L}P&= \lambda'(\sum_{i}F^{ii}-1)+u^{\frac{\alpha+1}{\alpha}}(F^{ij}h_{il}h_{jl}-F^{2})-u^{\frac{1}{\alpha}}F^{ij}\bar{R}_{\nu jli}\bar{g}(\lambda\partial_{r},e_{l}) \\
&\quad +\frac{1}{\alpha}\bar{g}(\lambda\partial_{r},\nabla \log u)-\frac{1}{\alpha}u^{\frac{\alpha+1}{\alpha}}F^{ij}\nabla_{i}\log u\nabla_{j}\log u.
\end{align*}
\end{lemma}

\begin{proof}
By equality \eqref{Eq:DDu}, we obtain
\begin{equation}\label{Eq:Lu}
\mathcal{L}u=\bar{g}(\lambda\partial_{r},\nabla F)+F^{ij}\bar{R}_{\nu jli}\bar{g}(\lambda\partial_{r},e_{l})+\lambda'F-uF^{ij}h_{il}h_{jl}.
\end{equation}

Moreover, using equation $F^{-\alpha}=u$,
\begin{align*}
\frac{\alpha}{\alpha+1}\mathcal{L}u^{\frac{\alpha+1}{\alpha}}&=u^{\frac{1}{\alpha}}\mathcal{L}u+\frac{1}{\alpha}u^{\frac{1-\alpha}{\alpha}}F^{ij}\nabla_{i}u\nabla_{j}u \\
&=u^{\frac{1}{\alpha}}\bar{g}(\lambda\partial_{r},\nabla F)+u^{\frac{1}{\alpha}}F^{ij}\bar{R}_{\nu jli}\bar{g}(\lambda\partial_{r},e_{l})+\lambda' \\
&\quad -u^{\frac{\alpha+1}{\alpha}}F^{ij}h_{il}h_{jl}+\frac{1}{\alpha}u^{\frac{1-\alpha}{\alpha}}F^{ij}\nabla_{i}u\nabla_{j}u.
\end{align*}

By equality \eqref{Eq:DDPhi} and equation $F^{-\alpha}=u$, we know
\begin{equation}\label{Eq:Lphi}
\mathcal{L}\Phi=F^{ij}(\lambda'\delta_{ij}-uh_{ij})=\lambda'\sum_{i}F^{ii}-uF=\lambda'\sum_{i}F^{ii}-u^{\frac{\alpha+1}{\alpha}}F^{2}.
\end{equation}

Combining \eqref{Eq:Lu} and \eqref{Eq:Lphi}, we obtain
\begin{align*}
\mathcal{L}P&=\mathcal{L}\Phi-\frac{\alpha}{\alpha+1}\mathcal{L}u^{\frac{\alpha+1}{\alpha}}\\
&= \lambda'(\sum_{i}F^{ii}-1)+u^{\frac{\alpha+1}{\alpha}}(F^{ij}h_{il}h_{jl}-F^{2})-u^{\frac{1}{\alpha}}F^{ij}\bar{R}_{\nu jli}\bar{g}(\lambda\partial_{r},e_{l}) \\
&\quad +\frac{1}{\alpha}\bar{g}(\lambda\partial_{r},\nabla \log u)-\frac{1}{\alpha}u^{\frac{\alpha+1}{\alpha}}F^{ij}\nabla_{i}\log u\nabla_{j}\log u.
\end{align*}

\end{proof}

\section{Proof of Theorem \ref{Th:H}}
\label{Sec:pfH}

For case $F=H$, tensor $F^{ij}=\frac{1}{n}\delta_{ij}$ and operator $\mathcal{L}=\frac{1}{n}\Delta$. Then \[ \sum_{i}F^{ii}=1, \] \[ F^{ij}h_{il}h_{jl}-F^{2}=\frac{1}{n}|h|^{2}-H^{2} \] and \[ F^{ij}\bar{R}_{\nu jli}\bar{g}(\lambda\partial_{r},e_{l})=\frac{1}{n}\overline{\mathrm{Ric}}(\nu,\lambda\partial_{r}^{T}), \] where $\overline{\mathrm{Ric}}$ denotes the Ricci curvature tensor of $\overline{M}$ and $\partial_{r}^{T}$ is the tangent part of $\partial_{r}$.

Thus Lemma \ref{Th:LP} gives
\begin{align}
\frac{1}{n}\Delta P&= u^{\frac{\alpha+1}{\alpha}}(\frac{1}{n}|h|^{2}-H^{2})-\frac{1}{n}u^{\frac{1}{\alpha}}\overline{\mathrm{Ric}}(\nu,\lambda\partial_{r}^{T})  \label{Eq:DeltaP} \\
&\quad +\frac{1}{\alpha}\bar{g}(\lambda\partial_{r},\nabla \log u)-\frac{1}{n\alpha}u^{\frac{\alpha+1}{\alpha}}|\nabla\log u|^{2}. \nonumber
\end{align}

Notice
\begin{align*}
\nabla P=\lambda\partial_{r}^{T}-u^{\frac{1}{\alpha}}\nabla u.
\end{align*}
Then
\begin{equation}\label{Eq:nablaP}
\bar{g}(\nabla P,\nabla \log u)=\bar{g}(\lambda\partial_{r},\nabla \log u)-u^{\frac{\alpha+1}{\alpha}}|\nabla\log u|^{2}.
\end{equation}

From \eqref{Eq:DeltaP} and \eqref{Eq:nablaP}, we obtain
\begin{align*}
\Delta P-\frac{n}{\alpha}\bar{g}(\nabla P,\nabla \log u)&=u^{\frac{\alpha+1}{\alpha}}(|h|^{2}-nH^{2})-u^{\frac{1}{\alpha}}\overline{\mathrm{Ric}}(\nu,\lambda\partial_{r}^{T}) \\
&\quad +\frac{n-1}{\alpha}u^{\frac{\alpha+1}{\alpha}}|\nabla\log u|^{2}.
\end{align*}
Lemma \ref{Th:Rivnu} shows $u^{\frac{1}{\alpha}}\overline{\mathrm{Ric}}(\nu,\lambda\partial_{r}^{T})\leq 0$. Then inequality $|h|^{2}-nH^{2}\geq 0$ indicates \[ \Delta P-\frac{n}{\alpha}\bar{g}(\nabla P,\nabla \log u)\geq \frac{n-1}{\alpha}u^{\frac{\alpha+1}{\alpha}}|\nabla\log u|^{2}\geq 0. \]

By the strong maximum principle, we know $P$ is constant. Then the above inequality shows $\nabla u=0$ for $n>1$. Consequently, $\nabla r=\partial_{r}^{T}=\frac{1}{\lambda}(\nabla P+u^{\frac{1}{\alpha}}\nabla u)=0$ everywhere in $M^{n}$. This implies $r$ is constant in $M^{n}$ which means $M^{n}$ is a slice.

\section{Proofs of Theorem \ref{Th:mainthm} and its corollaries}
\label{Sec:MainTHM}

\begin{proof}[Proof of Theorem \ref{Th:mainthm}]

From equality \eqref{Eq:Rnu}, 
\begin{align*}
F^{ij}\bar{R}_{\nu jli}\bar{g}(\lambda\partial_{r},e_{l})&=-\lambda\left(\frac{\lambda''}{\lambda}+\frac{c-\lambda'^{2}}{\lambda^{2}}\right)F^{ij}(\delta_{ij}r_{l}-\delta_{jl}r_{i}) r_{\nu}r_{l} \\
&=-u\left(\frac{\lambda''}{\lambda}+\frac{c-\lambda'^{2}}{\lambda^{2}}\right)\Big((\sum_{l}r_{l}^{2})(\sum_{i}F^{ii})-F^{ij}r_{i}r_{j}\Big).
\end{align*}

At any fixed point, we notice
\begin{equation}\label{Eq:r}
(\sum_{l}r_{l}^{2})(\sum_{i}F^{ii})-F^{ij}r_{i}r_{j}=(\sum_{l}r_{l}^{2})(\sum_{i}f^{i})-f^{i}r_{i}^{2}\geq 0,
\end{equation}
where $f^{i}:=\frac{\partial f}{\partial\kappa_{i}}>0$. And the equality occurs if and only if $\partial_{r}^{T}=0$.

Combining with the assumption \[\frac{\lambda''}{\lambda}+\frac{c-\lambda'^{2}}{\lambda^{2}}\geq 0,\] we obtain \[F^{ij}\bar{R}_{\nu jli}\bar{g}(\lambda\partial_{r},e_{l})\leq 0.\]

Thus, from Lemma \ref{Th:LP}, we have the following inequality
\begin{align*}
\mathcal{L}P&\geq \lambda'(\sum_{i}F^{ii}-1)+u^{\frac{\alpha+1}{\alpha}}(F^{ij}h_{il}h_{jl}-F^{2})+\frac{1}{\alpha}\bar{g}(\lambda\partial_{r},\nabla \log u) \\
&\quad -\frac{1}{\alpha}u^{\frac{\alpha+1}{\alpha}}F^{ij}\nabla_{i}\log u\nabla_{j}\log u.
\end{align*}

Since
\begin{align*}
\nabla P=\lambda\partial_{r}^{T}-u^{\frac{\alpha+1}{\alpha}}\nabla \log u,
\end{align*}
we have
\begin{align*}
u^{-\frac{\alpha+1}{\alpha}}\bar{g}(\lambda\partial_{r},\nabla P)=u^{-\frac{\alpha+1}{\alpha}}\lambda^{2}\sum_{i}r_{i}^{2}-\bar{g}(\lambda\partial_{r},\nabla \log u)
\end{align*}
and
\begin{align*}
F^{ij}\nabla_{i}\log u\nabla_{j}P=\lambda F^{ij}r_{j}\nabla_{i}\log u-u^{\frac{\alpha+1}{\alpha}}F^{ij}\nabla_{i}\log u\nabla_{j}\log u.
\end{align*}

Then we know
\begin{align*}
& \mathcal{L}P+\frac{1}{\alpha}u^{-\frac{\alpha+1}{\alpha}}\bar{g}(\lambda\partial_{r},\nabla P)-\frac{1}{\alpha}F^{ij}\nabla_{i}\log u\nabla_{j}P \\
&\qquad \geq \lambda'(\sum_{i}F^{ii}-1)+u^{\frac{\alpha+1}{\alpha}}(F^{ij}h_{il}h_{jl}-F^{2}) \\
&\hspace{3em} +\frac{1}{\alpha u}(\lambda^{2}F\sum_{i}r_{i}^{2}-\lambda F^{ij}r_{j}\nabla_{i}u).
\end{align*}

At any fixed point, choosing a frame such that $h_{ij}=\kappa_{i}\delta_{ij}$ and using equality \eqref{Eq:Du}, we have
\begin{align*}
\lambda^{2}F\sum_{i}r_{i}^{2}-\lambda F^{ij}r_{j}\nabla_{i}u=\lambda^{2}\sum_{i}(f-f^{i}\kappa_{i})r_{i}^{2}.
\end{align*}
We also know \[ \sum_{i}F^{ii}=\sum_{i}f^{i} \] and \[ F^{ij}h_{il}h_{jl}=f^{i}\kappa_{i}^{2}. \]

Thus, the following inequality holds
\begin{align*}
& \mathcal{L}P+\frac{1}{\alpha}u^{-\frac{\alpha+1}{\alpha}}\bar{g}(\lambda\partial_{r},\nabla P)-\frac{1}{\alpha}F^{ij}\nabla_{i}\log u\nabla_{j}P \\
&\qquad \geq \lambda'(\sum_{i}f^{i}-1)+u^{\frac{\alpha+1}{\alpha}}(f^{i}\kappa_{i}^{2}-f^{2})+\frac{\lambda^{2}}{\alpha u}\sum_{i}(f-f^{i}\kappa_{i})r_{i}^{2}.
\end{align*}
Assumption $\lambda'>0$ and Condition \ref{condtn} (v) imply \[ \mathcal{L}P+\frac{1}{\alpha}u^{-\frac{\alpha+1}{\alpha}}\bar{g}(\lambda\partial_{r},\nabla P)-\frac{1}{\alpha}F^{ij}\nabla_{i}\log u\nabla_{j}P\geq 0. \]
From $\frac{\partial f}{\partial \kappa_{i}}>0$ in $\Gamma$, we know $F^{ij}$ is positive definite. By the strong maximum principle, we know $P$ is constant. It indicates inequality \eqref{Eq:r} is actually equality. Then $\partial_{r}^{T}=0$ implies $M^{n}$ is a slice.
\end{proof}

\begin{proof}[Proof of Corollary \ref{Th:Hk}]
By Lemma 2.3 in \cite{Li-Wei-Xiong}, we know that principal curvatures $\kappa\in \Gamma_{k+1}\subset \widetilde{\Gamma}_{k}$ from $H_{k+1}>0$. From Example \ref{Ex:Hk} in Section \ref{Sec:prelim}, we finish the proof by letting $F=H_{k}^{\frac{1}{k}}$ and $\Gamma=\widetilde{\Gamma}_{k}$ in Theorem \ref{Th:mainthm}.
\end{proof}

\begin{proof}[Proofs of Corollary \ref{Th:K} and \ref{Th:F}]
See Example \ref{Ex:sigman} and \ref{Ex:candinc} in Section \ref{Sec:prelim} and use Theorem \ref{Th:mainthm}.
\end{proof}

\end{document}